\documentclass[12pt,reqno]{amsart}

\usepackage[english]{babel}
\usepackage{amsmath}
\usepackage{amsthm}
\usepackage{amssymb}
\usepackage{graphicx}
\usepackage{mathrsfs}  
\usepackage{xcolor}
\usepackage{xypic,verbatim}
\usepackage{parskip}
\usepackage{wasysym}
\usepackage[colorlinks]{hyperref}
\hypersetup{
    colorlinks,
    linkcolor={red!50!black},
    citecolor={green!50!black},
    urlcolor={blue!80!black}
}
\usepackage[nameinlink,capitalize]{cleveref}
\numberwithin{equation}{section}

\usepackage{cancel}

\usepackage{geometry}
\theoremstyle{plain}
\newtheorem{theorem}{Theorem}[section]
\newtheorem{proposition}[theorem]{Proposition}
\newtheorem{lemma}[theorem]{Lemma}   
\newtheorem{corollary}[theorem]{Corollary}

\newtheorem*{question*}{Question}

\theoremstyle{definition}

\newtheorem{notation}[theorem]{Notation}

\newtheorem{remark}[theorem]{Remark}

\newcommand{\SV}{\mathrm{SV}}

\newcommand{\rd}[1]{\left\lfloor #1\right\rfloor}
\newcommand{\ru}[1]{\left\lceil #1\right\rceil}

\newcommand{\oo}{\mathcal{O}}
\newcommand{\p}{\mathbb{P}}
\newcommand{\bbP}{\mathbb{P}}

\newcommand{\Q}{\mathbb{Q}}

\newcommand{\bbQ}{\mathbb{Q}}
\newcommand{\N}{\mathbb{N}}
\newcommand{\bbN}{\mathbb{N}}

\newcommand{\bfa}{\mathbf{a}}
\newcommand{\bfb}{\mathbf{b}}
\newcommand{\bfn}{\mathbf{n}}
\newcommand{\bfd}{\mathbf{d}}

\title{Non-defectivity of Segre-Veronese varieties}
\author[H. Abo]{Hirotachi Abo}
\address{Department of Mathematics, University of Idaho, Moscow, Idaho 83844--1103, United States of America}
\email{abo@uidaho.edu}
\author[M. C. Brambilla]{Maria Chiara Brambilla}
\address{ Universit\`a Politecnica delle Marche, Via Brecce Bianche, 60131 Ancona, Italy} 
\email{m.c.brambilla@univpm.it}

\author[F. Galuppi]{Francesco Galuppi}
\address{Faculty of Mathematics, Informatics, and Mechanics, University of Warsaw, Banacha 2, 02-097 Warsaw, Poland}
\email{galuppi@mimuw.edu.pl (ORCID 0000-0001-5630-5389)}

\author[A. Oneto]{Alessandro Oneto}
\address{Department of Mathematics, Universit\`a di Trento, Via Sommarive 14, 38123 Povo (Trento), Italy}
\email{alessandro.oneto@unitn.it}

\begin{document}

\begin{abstract}We prove that 
Segre-Veronese varieties are never secant defective 
if each %entry of the multi-degree 
degree is at least three. The proof is by induction on the number of factors, degree and dimension. %Our main technical tool is the differentiable Horace lemma.
As a corollary, we give an almost optimal non-defectivity result for Segre-Veronese varieties with one degree equal to one and all the others at least three.

% ---

% As an application, we give a further non-defectivity result for some Segre-Veronese varieties with lower degrees.

% \chiara{As an application, we give bounds for non-defectivity for Segre-Veronese varieties such that
%  one degree is one and each other %entry of the multi-degree 
% degree is at least three.}\\
% \chiara{chiara: I am happy also with the black abstract}
\end{abstract}

\maketitle

\section{Introduction}
A {\it Segre-Veronese variety}, the embedding of a multi-projective space by a very ample line bundle, parameterizes the rank-one partially symmetric tensors, and the compactification of the space parameterizing  those with partially symmetric rank at most $m$ is called the $m$-th secant variety of the Segre-Veronese variety. This paper concerns the problem of classifying the so-called \textit{defective} secant varieties of Segre-Veronese varieties, the ones with dimension smaller than expected. This problem is very classical and has its roots in XIX century algebraic geometry (see \cite{bernardi2018hitchhiker}). On the other hand, it is closely related to partially symmetric tensor rank, partially symmetric tensor border rank, simultaneous rank, and partially symmetric tensor decompositions, as well as their uniqueness, that are relevant topics to many branches of modern applied sciences%, statistics, and computer science
, see \cite{landsberg2011tensors}.
Hence, it has the potential to impact a variety of areas, including mathematics, computer science, and statistics. 

Our goal is to establish non-defectivity for a large family of Segre-Veronese varieties. The simplest examples of Segre-Veronese varieties are Veronese varieties, whose defectivity is completely understood %classification of defective secant varieties is 
due to the celebrated theorem by J.~Alexander and A.~Hirschowitz %\ale{They perfected an inductive method which allowed to show that the list of classically known defective cases in small dimensions and degrees was complete} 
\cite{AH}. Beyond Veronese varieties, this classification problem, however, is still far from complete. 

Some cases are better understood than the others. For example, the conjecturally complete list of defective secant varieties for Segre-Veronese varieties with two factors was suggested by M.~C.~Brambilla and H.~Abo in \cite{AboBra13}. Significant progress towards this conjecture was made by F.~Galuppi and A.~Oneto in \cite{GO}: they proved that if the bi-projective space is embedded by a linear system of degree at least three in both factors, then
% its linear system of a sufficiently large multi-degree, then 
its secant varieties are all non-defective. In this paper, we extend this result to an arbitrary number of factors.
% their result to Segre-Veronese varieties with more than two factors. 

M.~V.~Catalisano, A.~T.~Geramita, and A.~Gimigliano carried out the first systematic study of the secant varieties of Segre-Veronese varieties in two  papers \cite{CatGerGim05,CatGerGim08}. In these papers, they discovered many defective cases, including \textit{unbalanced cases} (where one of the factors of the multi-projective space has a much larger dimension than the rest). Several of these defective cases were later generalized by H.~Abo and M.~C.~Brambilla \cite{ABnewexamples}, as well as A.~Laface, A.~Massaranti, and R.~Richter \cite{LMR22}.

Regarding the secant non-defectivity, A.~Laface and E. Postinghel in
% , in their 2013 paper 
\cite{P1P1P1} employed toric approaches to show that the secant varieties of Segre-Veronese varieties of an arbitrary number of copies of the projective line are never defective. E. Ballico \cite{ballico} and  E.~Ballico, A.~Bernardi, and T.~Ma\'ndziuk \cite{ballico2023tensoring} proved non-defectivity for more families of Segre-Veronese varieties, with some assumptions on the dimensions. 

% In 2019, 
C.~Araujo, A.~Massarenti, and R.~Rischter \cite{AMR} developed a new approach using osculating projections and obtained an asymptotic bound under which the secant varieties of Segre-Veronese varieties always have the expected dimensions. Their bound was improved by A.~Laface, A.~Massarenti, and R.~Rischter \cite{LMR22}.
% in 2022. 

Very recently, A. Taveira Blomenhofer and A. Casarotti %, in their preprint
\cite{blomenhofer2023nondefectivity} significantly improved the 
% Laface-Massaranti-Rischter 
bound from \cite{LMR22} showing that most secant varieties of Segre-Veronese varieties are not defective% \ale{by generalizing a previous result by B. Ådlandsvick \cite{adlandsvik1988varieties} to a very general family of algebraic varieties which includes Segre-Veronese and most of varieties parametrizing tensors.}
.
However, there is still a range of values
 of $m$ 
for which the non-defectivity of the $m$-th secant variety of a Segre-Veronese variety is not known. As the longevity of the classification problem of the defective cases suggests, making this final stretch is difficult. The primary goal of this paper is to fill this gap for Segre-Veronese varieties embedded with degree at least three in all factors.
% moderately low multi-degrees. 
In the remaining part of this introduction, we introduce basic notation and state our main results.

For given $k$-tuples $\bfn = (n_1, n_2, \dots, n_k)$ and $\bfd = (d_1, d_2, \dots, d_k)$ of positive integers, we write $\mathbb{P}^\bfn$ for $\mathbb{P}^{n_1} \times \mathbb{P}^{n_2} \times \cdots \times \mathbb{P}^{n_k}$ and $\mathrm{SV}^\bfd_\bfn$ for the Segre-Veronese variety obtained by embedding $\mathbb{P}^\bfn$ by the morphism associated with its complete linear system $|\mathscr{O}_{\mathbb{P}^\bfn}(\bfd)|$. The closure of the union of secant $(m-1)$-planes to $\mathrm{SV}^\bfd_\bfn$ is called the $m$-th secant variety of $\mathrm{SV}^\bfd_\bfn$ and denoted by $\sigma_m(\mathrm{SV}^\bfd_\bfn)$. We say that it is non-defective if it is not $m$-defective for any positive integer $m$, that is, if $\sigma_m(\mathrm{SV}^\bfd_\bfn)$ has dimension equal to the expected one defined by a na\"ive parameter count, see \Cref{sec:tools} for explicit definitions. 
\begin{theorem}
\label{thm: ultimate goal}
    Let $k \geq 3$. If $d_1, d_2, \dots, d_k \geq 3$, then $\mathrm{SV}^\bfd_\bfn$ is not defective. 
\end{theorem}
The proof of this theorem, presented in \Cref{sec:base}, is an application of the \textit{differential Horace method}, which enables us to show the secant non-defectivity of a Segre-Veronese variety by induction on dimension and degree.
% , and the number of factors. 
This type of approach often leads to a complicated nested induction. The significance of this paper is to overcome this complication and to give a clean proof.% by applying induction on the number of factors $k$.%: \ale{in our case, the process falls into cases that are known by the aforementioned result by} 
% The proof uses the result by A. Taveira Blomenhofer and A. Casarotti~\cite{blomenhofer2023nondefectivity}. 

As a consequence of \Cref{thm: ultimate goal}, we deduce an almost optimal non-defectivity result for Segre-Veronese variety having one factor embedded in degree $1$ and all the others at least three.
% with a lower multi-degree is not defective.
\begin{theorem}
\label{thm:simultaneous_partialresult}
Let $k\ge 2$, let $\bfn = (n_1,n_2,\dots,n_k), \bfd = (d_1, d_2, \dots, d_k)$ be $k$-tuples of positive integers, let $|\bfn| = \sum_{i=1}^k n_i$, and let $N_{\bfn,\bfd} = \prod_{i=1}^k {n_i+d_i \choose n_i}$. If $d_1, d_2, \dots, d_k \ge 3$ and if 
    \[  
       m \le (n_0+1)\left\lfloor \frac{N_{\bfn,\bfd}}{n_0+|\bfn|+1}\right\rfloor  \quad \text{ or } \quad m \ge (n_0+1)\left\lceil \frac{N_{\bfn,\bfd}}{n_0+|\bfn|+1}\right\rceil,
    \]
    then $\SV_{(n_0,\bfn)}^{(1,\bfd)}$ is not $m$-defective.
\end{theorem}
The proof, presented in \Cref{sec:segre_induction}, is based on an inductive method which allows to deduce non-defectivity results for a Segre product $\bbP^{n}\times X$ from the non-defectivity of the algebraic variety $X$,
see \Cref{pro:segre induction}. It is worth noting that \Cref{thm:simultaneous_partialresult}
% presented in \Cref{thm: ultimate goal}
is stronger than \cite[Theorem~4.8]{blomenhofer2023nondefectivity} for these specific multidegrees, see \Cref{rem:multiple} for more details. 

While the rank of a tensor tells us about the length of a minimal decomposition as a sum of rank-one elements, identifiability is the uniqueness of such decomposition. For applied purposes it is very important to know when the Segre-Veronese variety is identifiable, namely when the general partially symmetric tensor has a unique decomposition. Thanks to \cite[Theorem 1.5]{MMidentifiability}, the non-defectivity of a variety has direct consequences on its identifiability, so we immediately get the following corollaries of \Cref{thm: ultimate goal,thm:simultaneous_partialresult}.
% state the following corollary of \Cref{thm: ultimate goal} and \Cref{thm:simultaneous_partialresult}.
\begin{corollary}
\label{corollary: identifiability}
Let $k\ge 2$, let $\bfn,\bfd \in \bbN^k$ be tuples of positive integers with $d_1, d_2, \dots, d_k\ge 3$, and let $N_{\bfn,\bfd} = \prod_{i=1}^k {n_i+d_i \choose n_i}$.  
\begin{enumerate}
\item If $m(|\bfn|+1)\le N_{\bfn,\bfd}$, 
then $\SV_\bfn^\bfd$ is $(m-1)$-identifiable.
\item If $m \le (n_0+1)\left\lfloor \frac{N_{\bfn,\bfd}}{n_0+|\bfn|+1}\right\rfloor$, then $\SV_{(n_0,\bfn)}^{(1,\bfd)}$ is $(m-1)$-identifiable. 
\end{enumerate}
\end{corollary}

% Despite \Cref{thm: ultimate goal} and \Cref{thm:simultaneous_partialresult}, the classification of defective Segre-Veronese varieties is not complete yet, because we are missing the cases in which one of the degrees is 1 or 2. Indeed, besides unbalanced cases, there are several known defective examples in the literature. However, all of them appear when the number of factors is 4 or less. This motivates us to close the section with the following  open problem.

% \chiara{In order to complete the classification of Segre-Veronese varieties, it remains to consider  the cases where the multidegree contains a 1 or a 2. Unfortunately in these cases there are several known defective examples in the literature. However, all of the defective cases, besides the unbalanced ones, appear when the number of factors is four or less. This motivates us to close the section with the following  open problem.}

In order to complete the classification of Segre-Veronese varieties, it remains to solve the cases in which one of the degrees is $1$ or $2$. A major difficulty here is that, besides the unbalanced ones, several defective cases are known, and it is complicated to shape a general inductive strategy that avoids them. We underline that all known balanced defective cases appear when the number of factors is four or less. For this reason, we want to explicitly draw attention to the following question. 
\begin{question*}
    Is it true that the only defective cases for Segre-Veronese varieties with at least five factors are the unbalanced cases?
\end{question*}

% \begin{problem}
%     Prove or disprove that there are no defective Segre-Veronese varieties with five or more factors, except for the unbalanced cases. 
% \end{problem}

During the final part of the preparation of the present manuscript, E.~Ballico privately informed us that he independently obtained a result similar to \Cref{thm: ultimate goal}% but, at the date of this submission, it has not appeared yet
.

\subsection*{Acknowledgements}
This project started during the semester AGATES in Warsaw,
and it is partially supported by  the Thematic Research Programme Tensors: geometry, complexity and quantum entanglement, University of Warsaw, Excellence Initiative – Research University and the Simons Foundation Award No. 663281 granted to the Institute of Mathematics of the Polish Academy of Sciences for the years 2021-2023.  

We would also like to thank the University of Trento and the Universit\`a Politecnica delle Marche, Ancona, for hosting us during our research visits.

M.~C.~Brambilla and A.~Oneto are members of INdAM-GNSAGA and have been funded by the European Union under NextGenerationEU. PRIN 2022 Prot. n. 2022ZRRL4C and Prot. n. 20223B5S8L, respectively. Views and opinions expressed are however those of the author(s) only and do not necessarily reflect those of the European Union or European Commission. Neither the European Union nor the granting authority can be held responsible for them.

\begin{center}
    \includegraphics[width=\textwidth]{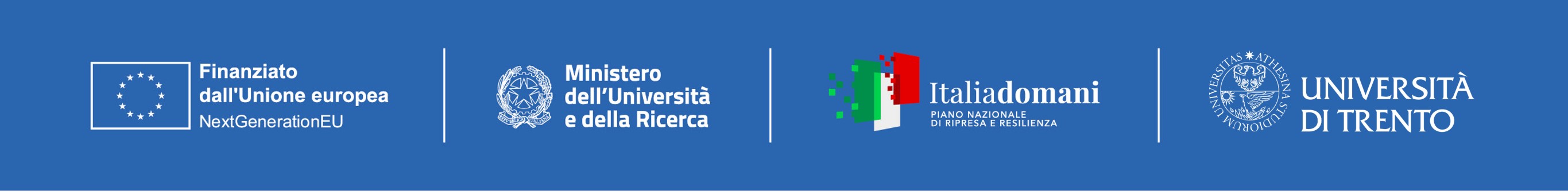}
\end{center}

\section{Tools and Background}\label{sec:tools}
We work over an algebraically closed field $\Bbbk$ of characteristic zero. 

Given an algebraic variety $X \subset \bbP^N$, the \emph{$m$-th secant variety}
\[
    \sigma_m(X) = \overline{\bigcup_{x_1,\ldots,x_m\in X} \langle x_1,\ldots,x_m \rangle} \subset \bbP^N.
\]
of $X \subset \bbP^N$ is the Zariski-closure of the union of all linear spaces spanned by $m$ points of $X$.

The notion of \emph{expected dimension} of $\sigma_m(X) \subset \bbP^N$ follows from a straightforward parameter count:
\[
    \exp.\dim\sigma_m(X) = \min\{N, m\dim(X)+m-1\}.
\]
It is immediate to see that this is always an upper bound for the actual dimension: we say that $X$ is \emph{$m$-defective} if $\dim\sigma_m(X) < \exp.\dim\sigma_m(X)$.

Let us fix some notation we will use throughout the paper.
\begin{notation}\label{definition: rup and rdown} 
    Let $\bfa = (a_1,\ldots,a_k), \bfb = (b_1,\ldots,b_k)\in \N^k$:
    \begin{itemize}
        \item For any $j \in \N$, we write $
            \bfa(j) = (a_1-j,a_2,\ldots,a_k)$.
        \item We write $\bfa \succeq \bfb$ if $a_i \geq b_i$ $\forall i \in \{1,\ldots,k\}$.
\item We write $|\bfa| = a_1+\dots+a_k$% \sum_{i=1}^k a_i$
.
    \end{itemize}
If $\bfn = (n_1,\dots,n_k)$ and $\bfd = (d_1,\dots,d_k)$ are $k$-tuples of positive integers, then we set
    \[
        N_{\bfn,\bfd} = \prod_{i=1}^k\binom{n_i+d_i}{d_i}
    \]
and we define
\begin{equation}\label{eq:critical_values}
        r^*(\bfn,\bfd) = \ru{\frac{N_{\bfn,\bfd}}{|\bfn|+1}} \quad \text{ and } \quad r_*(\bfn,\bfd) = \rd{\frac{N_{\bfn,\bfd}}{|\bfn|+1}}.
    \end{equation}
\end{notation}
\begin{remark}\label{rmk:critical_values}
The two values defined in \eqref{eq:critical_values} are \textit{critical} in the following sense:
\begin{itemize}
\item $r^*(\bfn,\bfd)$ is the smallest integer $m$ such that the $m$-th secant varieties is expected to fill the ambient space, namely it is expected to have dimension $N_{\bfn,\bfd}-1$. Since $\dim\sigma_m(X)$ is increasing with respect to $m$, if $\SV_\bfn^\bfd$ is not $r^*(\bfn,\bfd)$-defective then it is not $m$-defective for any $m \geq r^*(\bfn,\bfd)$.  For these values of $m$ we say that $\sigma_m(\SV_\bfn^\bfd)$ is \textit{superabundant}.
\item $r_*(\bfn,\bfd)$ is the largest integer $m$ such that the $m$-th secant varieties is expected to have dimension equal to the parameter count $m(|\bfn|+1)-1$. Since the difference of the dimensions of two consecutive secant varieties of $\SV_\bfn^\bfd$ is at most $|\bfn|+1$, if $\SV_\bfn^\bfd$ is not $r_*(\bfn,\bfd)$-defective then it is not $m$-defective for any $m \leq r_*(\bfn,\bfd)$. For these values of $m$ we say that $\sigma_m(\SV_\bfn^\bfd)$ is \textit{subabundant}.
%\item The $r_*(\bfn,\bfd)$-th secant variety corresponds to the \textit{last} one that is expected to have dimension equal to the parameter count $r_*(\bfn,\bfd)(|\bfn|+1)-1$. Since the difference of the dimensions of two consecutive secant varieties of $\SV_\bfn^\bfd$ is at most $|\bfn|+1$, if $\SV_\bfn^\bfd$ is not $r_*(\bfn,\bfd)$-defective, then it is not defective for any $r \leq r_*(\bfn,\bfd)$. These are called \textit{subabundant cases}.
 %   \item The $r^*(\bfn,\bfd)$-th secant variety corresponds to the \textit{first} one that is expected to fill the ambient space, i.e., to have dimension equal to $N_{\bfn,\bfd}$. Since a secant variety is contained in the following one, if $\SV_\bfn^\bfd$ is not $r^*(\bfn,\bfd)$-defective, then it is not defective for any $r \geq r_*(\bfn,\bfd)$. These are called \textit{superabundant cases}.
\end{itemize}
Therefore, in order to prove that a Segre-Veronese variety $\SV_\bfn^\bfd$ is never defective, it is enough to prove non-defectiveness at the critical values.
\end{remark}

% We follow the usual strategy of translating the problem to an {\it interpolation problem}. 

% Fixed $\bfn, \bfd \in \N^k$, let $\calL_{\bfn,\bfd}(Z)$ be the linear system of divisors in $\bbP^{\bfn}$ of multi-degree $\bfd$ containing the zero-dimensional scheme $Z$. If $Z = m_1P_1 + \ldots + m_sP_s$ is a union of fat points with general support, we write $\calL_{\bfn,\bfd}(m_1,\ldots,m_s) = \calL_{\bfn,\bfd}(Z)$. We write $m^s = (m,\ldots,m)$.

% It is well known that 
% \[
%     \codim \sigma_s(\SV^\bfd_\bfn) = \dim \calL_{\bfn,\bfd}(2^s).
% \]

% We say that $Z \subset \bbP^n$ is \emph{$\bfd$-regular}, or equivalently $\calL_{\bfn,\bfd}(Z)$ is \emph{regular}, if and only if $\dim\calL_{\bfn,\bfd} = \max\{0,N_{\bfn,\bfd}-\deg(Z)\}$. Given $Z' \subset Z \subset \bbP^N$ two zero-dimensional schemes, then:
% \begin{itemize}
%     \item if $Z$ is $\bfd$-regular then $Z'$ is also $\bfd$-regular;
%     \item if $\calL_{\bfn,\bfd}(Z')$ is empty, then $\calL_{\bfn,\bfd}(Z)$ is empty. 
% \end{itemize}
% These observations has the immediate consequence that, in order to prove \Cref{thm: ultimate goal} it is enough to consider the critical values given by \eqref{eq:critical_values}. In other words, it is enough to prove that 
% \begin{enumerate}
%     \item $\calL_{\bfn,\bfd}(2^{r_*(\bfn,\bfd)})$ is regular;
%     \item $\calL_{\bfn,\bfd}(2^{r^*(\bfn,\bfd)})$ is empty.
% \end{enumerate}

%\subsection*{Horace differential method}
The \emph{Horace method} is an inductive approach that goes back to G. Castelnuovo and was improved by J. Alexander and A. Hirschowitz, leading to the classification of defective Veronese varieties. This is a degeneration technique to study the dimension of complete linear systems of divisors with base points in general position with some multiplicities. We refer to \cite[section 2.2]{bernardi2018hitchhiker} for a detailed presentation of the method and its extensions. For the purpose of the present paper, we will employ the Horace method in the following formulation%recall the following precise statement of the method which highlight the inductive approach
.

% The general idea is as follows. Given a zero-dimensional scheme $Z$ and a divisor $H$, one considers the \emph{trace} ${\rm Tr}_H(Z)$ of $Z$ on $H$, given by the schematic intersection of $Z$ with $H$, and the \emph{residue} ${\rm Res}_H(Z)$ of $Z$ with respect to $H$, defined by $\mathcal{I}_Z : \mathcal{I}_H$, where $\mathcal{I}_Z, \mathcal{I}_H$ denote the defining ideal sheaves. The goal is to prove the regularity of $Z$ from the regularity of ${\rm Tr}_H(Z)$ and ${\rm Res}_H(Z)$. Due to arithmetic constrains this approach cannot always work. For this reason, Alexander and Hirschowitz optimized it, by refining the definition of residue and trace, giving rise to the so-called \emph{differential Horace method}. 

    %, which is formulated with the hypothesis $d_1\ge3$, but is true also when $d_1=2$.}
    %More precisely: the proof of \cite[Theorem 1.1]{AboBra13} is given in \cite[Theorem 2.9]{AboBra13}. The only time they require $d_1\ge 3$ is when they apply \cite[Lemma 2.5]{AboBra13}. That lemma is a rephrasing of \cite[Lemma 3]{chandler}. That assumption is not required. We only need $d_1\ge 2$, otherwise it makes no sense to take a divisor in the second step.}

\begin{theorem}[{\cite[Theorem~1.1]{AboBra13}}]\label{thm: differential horace}
Let $\bfn,\bfd \in \bbN^k$ be such that $d_1 \ge 3$ and let $r\in\N$. Let
\begin{align*}
    & s_r = s_r(\bfn,\bfd) = \rd{\frac{(|\bfn|+1)r-N_{\bfn,\bfd(1)}}{|\bfn|}} \quad \text{ and } \\
    & \epsilon_r = \epsilon_r(\bfn,\bfd) = (|\bfn|+1)r-N_{\bfn,\bfd(1)}-|\bfn|{s_r(\bfn,\bfd)}.
\end{align*} 
If all the following conditions are satisfied:
\begin{enumerate}
\item\label{item: n_1-1} $\SV_{\bfn(1)}^\bfd$ is not $s_r$-defective;
\item\label{item: d_1-1} $\SV_{\bfn}^{\bfd(1)}$ is not $(r-s_r)$-defective and $s_r \ge \epsilon_r$;
\item\label{item: d_1-2} $\SV_{\bfn}^{\bfd(2)}$ is not $(r-s_r-\epsilon_r)$-defective and $(r-s_r-\epsilon_r)(|\bfn|+1)\ge N_{\bfn,\bfd(2)}$;
\end{enumerate}
then $\SV_{\bfn}^{\bfd}$ is not $r$-defective.
\end{theorem}

\begin{remark}
    The numerical assumption in condition \eqref{item: d_1-1} guarantees that $\sigma_{r-s_r}(\SV_{\bfn}^{\bfd(1)})$ is a subabundant case.
    % , i.e., $(r-s_r)(|\bfn|+1) \leq N_{\bfn,\bfd(1)}$
     On the contrary, the numerical condition \eqref{item: d_1-2} affirms that $\sigma_{r-s_r-\epsilon_r}(\SV_{\bfn}^{\bfd(2)})$ is a superabundant case.
\end{remark}

%\subsection*{Taveira Blomenhofer-Casarotti general result}
While it may be difficult to prove that a variety is not defective, %Most of the known defective algebraic varieties are $m$-defective for values of $m$ that are close to the critical ones. For this reason, 
in the literature several varieties have been proven to be not $m$-defective when $m$ is sufficiently far from the critical ones. One example is \cite{blomenhofer2023nondefectivity}, where A.~Blumenhofer and A.~Casarotti generalize a result by B. Ådlandsvik \cite{adlandsvik1988varieties} and prove non-defectivity for varieties that are invariant under the action of a group $G$ and contained in irreducible $G$-module.
%
%\ale{For instance A. Taveira Blomenhofer and A. Casarotti recently announced a beautiful result which provides a uniform interval around the critical values in which all defective cases are collected whenever we consider subvarieties invariant with respect to the action of a group $G$ and contained in irreducible $G$-module, like the Segre-Veronese variety $\SV_\bfn^\bfd$ with respect to $G = {\rm GL}_{n_1}\times\cdots\times{\rm GL}_{n_k}$, see \cite{blomenhofer2023nondefectivity}. This result generalizes a previous analogous result by B. Ådlandsvik for projective varieties having most secant that are not projective cones, see~\cite{adlandsvik1988varieties}.}
%
% In order to check condition \eqref{item: d_1-2} we will use the bound recently obtained
% by Blomenhofer-Casarotti.  
%
The precise statement that we need in the case of Segre-Veronese varieties is the following.

\begin{theorem}[{\cite[Theorem 4.8]{blomenhofer2023nondefectivity}}] \label{BC} Let $\bfn,\bfd\in\N^k$% be $k$-tuples of positive integers
. If $    m \le r_*(\bfn,\bfd) - |\bfn| - 1$ or $m \ge r^*(\bfn,\bfd)+|\bfn|+1$,
then $\SV_{\bfn}^\bfd$ is not $m$-defective.
\end{theorem}

\section{Proof of \Cref{thm: ultimate goal}}\label{sec:base}

We prove \Cref{thm: ultimate goal} by induction on the number $k$ of factors, on the dimension $n_1$ and on the degree $d_1$% of the factor with smallest dimension which, up to reordering the factors, we may assume to be the first one
. First we give the necessary results to deal with the base case for $n_1 = 1$ and two base cases for $d_1 \in \{3,4\}$.
% We prove \Cref{thm: ultimate goal} by induction on the number of factors $k$. The base case $k = 2$ has been solved in \cite[Theorem 1.2]{GO}. %by the last two authors.\begin{theorem}[{\cite[Theorem 1.2]{GO}}]\label{thm:GO}    Let $\bfn$ and $\bfd \succeq (3,3)$ be pairs of positive integers. Then $\SV_\bfn^\bfd$ is not defective.\end{theorem}
% %
% Fixed $k \geq 3$, we apply \Cref{thm: differential horace} and we proceed by double induction on the dimension and the degree. For that, we need one base case for $n_1 = 1$ and two base cases for $d_1 \in \{3,4\}$. 
%
%\subsection*{Initial case with respect to the number of factors} The case $k = 2$ of \Cref{thm: ultimate goal} was proven by the last two authors.
%
%
%\subsection{Initial case with respect to dimension} 

We recall the following result by E. Ballico.

\begin{theorem}[{\cite[Theorem~2]{ballico}}]
    \label{pro: ballico adds a factor} 
    Let $X\subset \p^N$ be an irreducible non-degenerate variety with $\dim(X)\ge 3$.  Let 
    \[
        r=\rd{\frac{N+1}{\dim(X)}}
    \]
    and assume that $X$ is not $r$-defective. Let $d\ge 2$ and consider $Y=\p^1\times X$ embedded in $\p^{(d+1)(N+1)-1}$ by the line bundle $\oo_{\p^1}(d)\otimes\oo_X(1)$. If $N+1> \dim(X)^2$, then $Y$ is not defective.
\end{theorem}

By applying this theorem, we prove a technical result that will be useful in the proof of \Cref{thm: ultimate goal}.% holds for $n_1 = 1$. %of  is an corollary of the latter proposition by induction on the number of factors $k$, where the initial case $k = 2$ was proven in \cite[Theorem 1.2]{GO}.

% \begin{equation}\label{eq: induction hypothesis on k}
%     \mbox{if }d_i\ge 3\mbox{ for every $2\le i\le k$, then }\SV_{n_2\times\dots\times n_k}^{d_2,\dots,d_k}\mbox{ is not defective}.
% \end{equation}

% Let us start by dealing with the case $n_1=1$.

\begin{corollary}\label{lemma: n1 = 1}
    Let $k\ge 2$. Let $\bfn' = (n_2,\ldots,n_k)$ and $\bfd' = (d_2,\ldots,d_k) \succeq 3^{k-1}$ be $(k-1)$-tuples of positive integers. If $\SV_{\bfn'}^{\bfd'}$ is not defective and $d_1\ge 2$, then $\SV_{(1,\bfn')}^{(d_1,\bfd')}$ is never defective.
%  \begin{proof}
%  This follows by \Cref{pro: ballico adds a factor} and assumption \eqref{eq: induction hypothesis on k}.
%  \end{proof}
\end{corollary}
\begin{proof}
We start by proving that 
\begin{align}
        \prod_{i=2}^k\binom{n_i+d_i}{n_i}>(n_2+\dots+n_k)^2.\label{ballico numbers}
    \end{align}
The left-hand-side is increasing with respect to $d_2,\dots ,d_k$, so it is enough to prove \eqref{ballico numbers} for {$\bfd=3^{k-1}$.}
%$d_1 = \ldots = d_k = 3$. 
%Now, in the expansion of $\prod_{i=2}^k\binom{n_i+3}{n_i} - (n_2+\ldots+n_k)^2$ as a polynomial in $\bbQ[n_1,\ldots,n_k]$ all the coefficients are positive. Indeed, i
On the left-hand-side %n $\prod_{i=2}^k\binom{n_i+3}{n_i}$ we have that 
$n_i^2$ appears with coefficient $1$, while $n_in_j$ ($i \neq j$) appears with coefficient $11^2/6^2 > 2$.

%We prove our statement by induction on $k$. Suppose that $k = 2$. 
If $k=2$ and $n_2=n_3=1$, then $\SV_{(1,1,1)}^{(d_1,d_2,d_3)}$ is not defective by \cite[Theorem 3.1]{P1P1P1}. In any other case we have $n_2+\dots +n_k\ge 3$, so we can apply \Cref{pro: ballico adds a factor} to the variety $X=\SV_{\bfn'}^{\bfd'}$, which is not defective by hypothesis.%: indeed, the latter is not defective by \cite[Theorem 1.2]{GO} and the numerical condition holds by \eqref{ballico numbers}.
%We apply  to deduce that $\SV_{1,n_2,n_3}^{d_1,d_2,d_3}$ is not defective which is known to be non-defective by \cite[Theorem 1.2]{GO}. Since \Cref{pro: ballico adds a factor} applies to varieties at least three-dimensional, the case $n_2 = n_3 = 1$ should be done differently, but this case is know by 
%    %
%For $k \geq 4$, the variety $\SV_{\bfn'}^{\bfd'}$ has dimension at least $3$. Since, we assume that $\SV_{\bfn'}^{\bfd'}$ is not defective by assumption and we have already proved \eqref{ballico numbers}, then the statement follows \Cref{pro: ballico adds a factor}.
\end{proof}

%We continue with the proof of the general case of \Cref{thm: ultimate goal}. As already mentioned, our main strategy will be to use the Horace differential method. 

%%%%%%%%%%%%%%%%%%%%%%%%%%%%%%%%%%%%%%%%%%%%%%%%%%%%%%%%%%%%%%%%%%%%%%%%%%%%%%%%%%%%%%%%%%%%%%%%%%%%%%%%%%%%%%%%%%%%%%%%%%%%%%
%%%%%%%%%%%%%%%%%%%%%%%%%%%%%%%%%%%%%%%%%%%%%%%%%%%%%%%%%%%%%%%%%%%%%%%%%%%%%%%%%%%%%%%%%%%%%%%%%%%%%%%%%%%%%%%%%%%%%%%%%%%%%%
%%%%%%%%%%%%%%%%%%%%%%%%%%%%%%%%%%%%%%%%%%%%%%%%%%%%%%%%%%%%%%%%%%%%%%%%%%%%%%%%%%%%%%%%%%%%%%%%%%%%%%%%%%%%%%%%%%%%%%%%%%%%%%
%\subsection{Initial cases with respect to degree}\label{ssec:initial_cases_degree}
%\section{The base cases in small degrees}
In \Cref{thm: differential horace} there are no assumptions about the order of the $d_i$'s and the $n_i$'s. %Therefore, in the proof of \Cref{thm: ultimate goal}, we decide to apply the Horace method to the factor with smallest dimension $n_i$. For this reason, from now on we will assume that the factors are ordered in such a way that 
Up to permuting the  factors, it is not restrictive to suppose that $n_1 \leq n_2 \leq \cdots \leq n_k$. This is crucial for some of our numerical computations. 

Now we deal with the case $d_1=3$. In order to make our arguments easier to read, we postpone some of the arithmetic computations to \Cref{appendix: contacci}.

% We prove the base cases of the induction for small values of $d_1$. 
% By symmetry of the $d_i$'s, and without loss of generality, we will apply \Cref{thm: differential horace} under the assumption that $n_1 \leq n_2 \leq \cdots \leq n_k$. We start by proving the base cases on small degrees. 

% In the next proposition, we prove that \Cref{thm: ultimate goal} holds for $d_1=3$
% %\label{sec:d1=3}

\begin{proposition}\label{pro: initial case d1 = 3}
Let $k\ge 2$ and let $\bfn' = (n_2 \leq n_3 \le \cdots \leq n_k)$ and $\bfd' = (d_2,d_3,\ldots,d_k) \succeq 3^{k-1}$ be $(k-1)$-tuples of positive integers. If $\SV_{\bfn'}^{\bfd'}$ is not defective and  $n_1$ is a positive integer, then $\SV_{(n_1,\bfn')}^{(3,\bfd')}$ is not defective.
\end{proposition}
\begin{proof}

We argue by induction on $n_1$. The initial case $n_1=1$ is \Cref{lemma: n1 = 1}. For $n_1 \geq 2$, we prove that $\SV_\bfn^\bfd$ is not $r$-defective for the critical values $r\in\{r_*(\bfn,\bfd),r^*(\bfn,\bfd)\}$ by applying \Cref{thm: differential horace}. We check all conditions: 
    \begin{enumerate}
\item $\SV_{\bfn(1)}^\bfd$ is not defective by our inductive hypothesis on $n_1$.
\item By \Cref{lemma:d1=3_subabundant_BC}, $r-s_r \leq r_*(\bfn,\bfd(1))-|\bfn|-1$, so $\SV_{\bfn}^{\bfd(1)}$ is not $(r-s_r)$-defective by \Cref{BC}. The numerical condition \Cref{thm: differential horace}(2) is \Cref{lemma:s_vs_epsilon}.
\item By \Cref{lemmaBC}, $r-s_r-\epsilon_r \geq r^*(\bfn,\bfd(2))+|\bfn|+1$, so $\SV_{\bfn}^{\bfd(2)}$ is not $(r-s_r-\epsilon_r)$-defective by \Cref{BC}. The numerical condition \Cref{thm: differential horace}(3) also follows from \Cref{lemmaBC}, because $r^*(\bfn,\bfd(2)) + |\bfn| + 1 \geq \frac{N_{\bfn,\bfd(2)}}{|\bfn|+1}$. \qedhere
    \end{enumerate} 
    % The second condition follows from , 
    % \chiara{is true by Lemma \ref{lemma:d1=3_subabundant_BC}}
    %  The third condition is true by \Cref{pro:segre induction} \ale{-- no need: it is BC} since $\bfd(2) = (1,d_2,\ldots,d_k)$, while the numerical condition follows from \Cref{lemmaBC} because $r^*(\bfn,\bfd(2)) + |\bfn| + 1 \geq \frac{N_{\bfn,\bfd(2)}}{|\bfn|+1}$.
\end{proof}

Next we consider  the case $d_1=4$. The proof is very similar to the previous one. The only difference is that we apply \Cref{pro: initial case d1 = 3} to check the second condition in the \Cref{thm: differential horace}.

\begin{proposition}\label{pro: initial case d1 = 4}
Let $k\ge 2$ and let $\bfn' = (n_2 \leq n_3 \le \cdots \leq n_k)$ and $\bfd' = (d_2,d_3,\ldots,d_k) \succeq 3^{k-1}$ be $(k-1)$-tuples of positive integers. If $\SV_{\bfn'}^{\bfd'}$ is not defective and  $n_1$ is a positive integer, then $\SV_{(n_1,\bfn')}^{(4,\bfd')}$ is not defective.
%    Let $k\ge 2$. Let $\bfn = (n_1 \leq n_2 \le\cdots \leq n_k)$ and $\bfd = (4,d_2,\ldots,d_k) \succeq (4,3^{k-1})$ be $k$-tuples of positive integers. Then $\SV_\bfn^\bfd$ is never defective.
\end{proposition}
\begin{proof}
 We argue by induction on $n_1$. The initial case $n_1=1$ is \Cref{lemma: n1 = 1}. For $n_1 \geq 2$, we prove that $\SV_\bfn^\bfd$ is not $r$-defective for $r\in\{r_*(\bfn,\bfd),r^*(\bfn,\bfd)\}$ by applying \Cref{thm: differential horace}. We check all conditions:
    \begin{enumerate}
 \item $\SV_{\bfn(1)}^\bfd$ is not defective by the inductive hypothesis.
\item $\SV_{\bfn}^{\bfd(1)}$ is not defective by \Cref{pro: initial case d1 = 3}. The numerical condition \Cref{thm: differential horace}(2) is \Cref{lemma:s_vs_epsilon}. 
\item By \Cref{lemmaBC}, we have $r-s_r-\epsilon_r \geq r^*(\bfn,\bfd(2))+|\bfn|+1$, so $\SV_{\bfn}^{\bfd(2)}$ is not $(r-s_r-\epsilon_r)$-defective by \Cref{BC}. The numerical condition \Cref{thm: differential horace}(3) also follows from \Cref{lemmaBC}, because $r^*(\bfn,\bfd(2)) + |\bfn| + 1 \geq \frac{N_{\bfn,\bfd(2)}}{|\bfn|+1}$. \qedhere
    \end{enumerate}
     % The second condition is , while the numerical requirement has been already proven in \Cref{lemma: main condition s epsilon}     \chiara{is true by Lemma \ref{lemma:d1=3_subabundant_BC}}. The third condition follows from \Cref{lemmaBC}: indeed, it allows us to deduce that $\SV_{\bfn}^{\bfd(2)}$ is not $(r-s_r-\epsilon_r)$-defective by \Cref{BC}, while it also guarantees the numerical requirement of the third condition of \Cref{thm: differential horace} because $r^*(\bfn,\bfd(2)) + |\bfn| + 1 \geq \frac{N_{\bfn,\bfd(2)}}{|\bfn|+1}$.
\end{proof}

%%%%%%%%%%%%%%%%%%%%%%%%%%%%%%%%%%%%%%%%%%%%%%%%%%%%%%%%%%%%%%%%%%%%%%%%%%%%%%%%%%%%%%%%%%%%%%%%%%%%%%%%%%%%%%%%%%%%%%%%%%%%%%
%%%%%%%%%%%%%%%%%%%%%%%%%%%%%%%%%%%%%%%%%%%%%%%%%%%%%%%%%%%%%%%%%%%%%%%%%%%%%%%%%%%%%%%%%%%%%%%%%%%%%%%%%%%%%%%%%%%%%%%%%%%%%%
%%%%%%%%%%%%%%%%%%%%%%%%%%%%%%%%%%%%%%%%%%%%%%%%%%%%%%%%%%%%%%%%%%%%%%%%%%%%%%%%%%%%%%%%%%%%%%%%%%%%%%%%%%%%%%%%%%%%%%%%%%%%%%

\subsection*{Proof of \Cref{thm: ultimate goal}}\label{sec:main} 
As we pointed out, without loss of generality we may assume that $n_1 \leq n_2 \leq \cdots \leq n_k$. We argue by induction on $k \geq 2$. The base case $k = 2$  is \cite[Theorem 1.2]{GO}. We assume that $k \geq 3$ and that $\SV_{n_2,\dots,n_k}^{d_2,\dots,d_k}$ is not defective, and we prove that $\SV_{\bfn}^{\bfd}$ is not defective. %We apply \Cref{thm: differential horace} on the factor of smallest dimension; in other words %, by symmetry of the $d_i$'s, we assume 
We proceed by one-step induction on $n_1$ and by two-step induction on $d_1$. The base case $n_1 = 1$ is a consequence of \Cref{lemma: n1 = 1}, while the base cases $d_1 \in \{3,4\}$ follow from \Cref{pro: initial case d1 = 3,pro: initial case d1 = 4}. Now we suppose that $n_1\ge 2$, $d_1\ge 5$ and we assume that $\SV_{\bfn(1)}^\bfd$, $\SV_{\bfn}^{\bfd(1)}$ and $\SV_{\bfn}^{\bfd(2)}$ are not defective. Thanks to \Cref{thm: differential horace}, in order to conclude, it is enough to show that the two numerical conditions hold. The numerical condition of \Cref{thm: differential horace}(2) is \Cref{lemma:s_vs_epsilon}.
The numerical condition \Cref{thm: differential horace}(3) follows from \Cref{lemmaBC}, because $r^*(\bfn,\bfd(2)) + |\bfn| + 1 \geq \frac{N_{\bfn,\bfd(2)}}{|\bfn|+1}$. \hfill \qedsymbol

%%%%%%%%%%%%%%%%%%%%%%%%%%%%%%%%%%%%%%%%%%%%%%%%%%%%%%%%%%%%%%%%%%%%%%%%%%%%%%%%%%%%%%%%%%%%%%%%%%%%%%%%%%%%%%%%%%%%%%%%%%%%%%
%%%%%%%%%%%%%%%%%%%%%%%%%%%%%%%%%%%%%%%%%%%%%%%%%%%%%%%%%%%%%%%%%%%%%%%%%%%%%%%%%%%%%%%%%%%%%%%%%%%%%%%%%%%%%%%%%%%%%%%%%%%%%%
%%%%%%%%%%%%%%%%%%%%%%%%%%%%%%%%%%%%%%%%%%%%%%%%%%%%%%%%%%%%%%%%%%%%%%%%%%%%%%%%%%%%%%%%%%%%%%%%%%%%%%%%%%%%%%%%%%%%%%%%%%%%%%
\section{The splitting lemma and proof of \Cref{thm:simultaneous_partialresult}}\label{sec:segre_induction}

Let $V$ and $W$ be $\Bbbk$-vector spaces with $\dim V = n_0 + 1$ and $\dim W = \alpha + 1$. Let $X \subseteq \mathbb{P}W$ be a non-degenerate algebraic variety and let $Y = \bbP V \times X \subset \bbP (V\otimes W)$ be the Segre product. In this section we describe an inductive method useful to derive non-defectivity of $Y$ from the non-defectivity of $X$. Using such a method we prove \Cref{thm:simultaneous_partialresult}.

% Let $V$ be an $(\ell+1)$-dimensional vector space over $\Bbbk$ and let $W$ be an $(\alpha+1)$-dimensional vector space over $\Bbbk$. We denote by $\mathbb{P} V$ the projective space of $V$ and by $\mathbb{P} W$ the projective space of $W$. Let $X \subseteq \mathbb{P}W$ be a non-degenerate projective variety with inclusion $\varphi: X \rightarrow \mathbb{P} W$ and let $\mathscr{O}_X(1) = \varphi^* \mathscr{O}_{\mathbb{P} W}(1)$. For each $i \in \{1,2\}$, we write $\pi_i$ for the projection of $\mathbb{P} V \times X$ to its $i$th factor, and let $Y$ be the image of $\mathbb{P}V \times X$ in $ \mathbb{P}(V \otimes W)$ under the embedding associated with the very ample line bundle $\mathscr{O}_{\mathbb{P}V}(1) \boxtimes \mathscr{O}_X(1) = \pi_1^* \mathscr{O}_{\mathbb{P}V}(1) \otimes \pi_2^* \mathscr{O}_X(1)$. The purpose of this subsection is to prove the following proposition: 
Let $\widehat{T}_pY$ denote the affine cone over the tangent space to $Y$ at $p$. It is immediate to observe that, if $p = [v \otimes w] \in Y$, then 
$
    \widehat{T}_pY = V \otimes w + v \otimes \widehat{T}_{[w]} X. 
$

\begin{proposition}[Segre induction]\label{pro:segre induction} 
    Let
    \[
        a_*=\rd{\frac{\alpha+1}{n_0+\dim X+1}}\mbox{ and } a^*=\ru{\frac{\alpha+1}{n_0+\dim X+1}}.
    \]
    Suppose that $X$ is neither $a_*$-defective nor $a^*$-defective. If $m \le (n_0+1)a_*$ or $m \ge (n_0+1)a^*$, then $Y$ is not $m$-defective.
\end{proposition}

By combining \Cref{pro:segre induction} and \Cref{thm: ultimate goal} we immediately deduce \Cref{thm:simultaneous_partialresult}. %Indeed,by  we know that $ \SV^\bfd_\bfn$ is never defective.  By applying \Cref{pro:segre induction} we get that  $\SV_{(n_0,\bfn)}^{(1,\bfd)}$ is not $m$-defective.
% Since $$r_*((n_0,\bfn),(1,\bfd))-(n_0+1)\left\lfloor \frac{N_{\bfn,\bfd}}{n_0+|\bfn|+1}\right\rfloor\le n_0$$
%$$ and
% $$(n_0+1)\left\lceil\frac{N_{\bfn,\bfd}}{n_0+|\bfn|+1} \right\rceil - r_*((n_0,\bfn),(1,\bfd)) \le n_0$$
% we deduce the second statement.
The rest of this section is devoted to proving \Cref{pro:segre induction}. We will employ the so-called splitting lemma, which is a variation of the inductive approach successfully employed in studying secant varieties of various classically known varieties such as Segre varieties \cite{burgisser2013algebraic,abo2009induction} and Segre-Veronese varieties with two factors embedded in bi-degree $(1,2)$~\cite{AboBra09}. 
The splitting lemma is based on the classical {\it Terracini's lemma}%, see e.g.\ \cite{bernardi2018hitchhiker}
. 
\begin{lemma}[{Terracini's Lemma, \cite{terracini1911sulle}}]
    Let $Z \subset \bbP^N$ be a non-degenerate algebraic variety. Let $p_1,\ldots,p_m \in Z$ be generic and let $q \in \langle p_1,\ldots,p_m \rangle$ be generic. Then, 
    \[
        \widehat{T}_q\sigma_m(Z) = \sum_{i=1}^m \widehat{T}_{p_i}Z.
    \]
\end{lemma}

% , which states that problem of the dimension of the $s$th secant variety $\sigma_s(Y)$ of $Y$ is equivalent to the problem of finding the dimension of the vector space sum $\sum_{i=1}^s \widehat{T}_{p_i} Y$ of the affine cones $\widehat{T}_{p_1} Y, \widehat{T}_{p_2} Y,\dots, \widehat{T}_{p_s} Y$ over the tangent spaces to $Y$ at $s$~points $p_1, p_2, \dots, p_s$:  
% \[
% \dim \sigma_s(Y) = \dim \sum_{i=1}^s  \widehat{T}_{p_i} Y -1.
% \]

% The splitting lemma uses the method of specializing the $s$ points $p_1, p_2, \dots, p_s$, which is based on the following multilinear description of the affine cone $\widehat{T}_p Y$ over the tangent space to $Y$ at a point $p  = [v\otimes w]\in Y$: 
% \[
% \widehat{T}_p Y = V \otimes w + v \otimes \widehat{T}_{[w]} X. 
% \] 
\begin{notation}
    Fixed $X$ and $Y$ as above, we denote by $T(n_0,s,t)$ the following property:
    \begin{center}
        {\it For generic $p_1,\ldots,p_m, q_1,\ldots,q_t \in Y$, with $p_i = [v_i\otimes w_i]$ and $q_i = [v_i'\otimes w_i']$, \\
        $\dim \left(\sum_{i=1}^m  \widehat{T}_{p_i} Y + \sum_{i=1}^t V \otimes w_i'\right) = \min \{(n_0+1)(\alpha+1), m\, (n_0+x+1) + t \, (n_0+1)\}$.}
    \end{center}
    Moreover, analogously to the terminology introduced in \Cref{sec:tools}, we say that the triple $(n_0,m,t)$ is \textit{subabundant} if $m(n_0+\dim X+1)+t(n_0+1) \leq (n_0+1)(\alpha+1)$; while we say that it is \textit{superabundant} if $m(n_0+\dim X+1)+t(n_0+1) \geq (n_0+1)(\alpha+1)$.
\end{notation}

% \begin{definition}
% For each $i \in \{1, 2, \dots, s\}$ and $j \in \{1, 2, \dots, t\}$, let $p_i = [v_i \otimes w_i] \in Y$ and let $q_j = [v_j' \otimes w_j'] \in Y$. Consider the following subspace of $V \otimes W$:  
% \[
% \sum_{i=1}^s  \widehat{T}_{p_i} Y + \sum_{i=1}^t V \otimes w_i' = \sum_{i=1}^s (V \otimes w_i + v_i \otimes \widehat{T}_{[w_i]} X) + \sum_{i=1}^t V \otimes w_i'. 
% \]

% By definition, the dimension of this subspace is bounded above by the smallest integer between $(\ell+1)(\alpha+1)$ and $s\, (\ell+m+1) + t \,(\ell+1)$. We say that $T(\ell,s,t)$ is {\it true} if $\sum_{i=1}^s  \widehat{T}_{p_i} Y + \sum_{i=1}^t V \otimes w_i'$ has the ``expected'' dimension, that is,  
% \[
% \dim \left(\sum_{i=1}^s  \widehat{T}_{p_i} Y + \sum_{i=1}^t V \otimes w_i'\right) = \min \{(\ell+1)(\alpha+1), s\, (\ell+m+1) + t \, (\ell+1)\},
% \] 
% for a generic choice of $p_1, \dots, p_s, q_1, \dots, q_t \in Y$. 
% \end{definition}

\begin{remark}\label{rmk:subabundant_superabundant}
% The $s$th secant variety $\sigma_s(Y)$ of $Y$ has the expected dimension if and only if $T(\ell,s,0)$ is true. 
By Terracini's Lemma, the property $T(n_0,m,0)$ is equivalent to say that $Y$ is not $m$-defective. For example, \Cref{rmk:critical_values} can be rephrased by saying that:
\begin{itemize}
    \item $T(n_0,m,0)$ implies $T(n_0,m',0)$ for every $m' \leq m$ whenever $(n_0,m,0)$ is subabundant;
    \item $T(n_0,m,0)$ implies $T(n_0,m',0)$ for every $m' \geq m$ whenever $(n_0,m,0)$ is superabundant.
\end{itemize}
\end{remark}

% \begin{notation}
% We call the triple $(\ell,s,t)$ {\it subabandant} if $(\ell+1)(\alpha+1) \geq s\, (\ell+m+1) + t \, (\ell+1)$; while we call $(\ell,s,t)$ {\it superabandant} if $(\ell+1)(\alpha+1) \leq s\, (\ell+m+1) + t \, (\ell+1)$. 
% \end{notation}

% \begin{remark}
% \label{rem:abandance}
% Suppose that $T(\ell,s,0)$ is true. An application of Terracini's lemma shows that if $(\ell,s,0)$ is subabandant, then $T(\ell,s',0)$ is true for every $s' \leq s$, while if $(\ell,s,0)$ is superabandant, then $T(\ell,s',0)$ is true for every $s' \geq s$. 
% \end{remark}

% \begin{lemma}[Splitting lemma]
% \label{lem:Segre_induction}
% Let $m' \in \{0,1, \dots, m\}$ and let $n' \in \{0, 1, \dots, n_0-1\}$. If $(\ell',s',t+s-s')$ and $(\ell-\ell'-1,s-s',t+s')$ are both subabandant (resp. subabandant) and if $T(\ell',s',t+s-s')$ and $T(\ell-\ell'-1,s-s',t+s')$ are both true, then $(\ell,s,t)$ is subabandant (resp. superabandant) and $T(\ell,s,t)$ is true. 
% \end{lemma}

\begin{lemma}[Splitting Lemma]
\label{lem:Segre_induction}
Let $m' \in \{0,1, \dots, m\}$ and let $n' \in \{0, 1, \dots, n_0-1\}$.
\begin{enumerate}
    \item If $(n',m',t+m-m')$ and $(n_0-n'-1,m-m',t+m')$ are both subabandant (resp. subabandant) then $(n_0,m,t)$ is subabandant (resp. superabandant).
    \item If $T(n',m',t+m-m')$ and $T(n_0-n'-1,m-m',t+m')$ are both true, then $T(n_0,m,t)$ is true.
\end{enumerate}
\end{lemma}

\begin{proof}
(1) If $(n',m',t+m-m')$ and $(n_0-n'-1,m-m',t+m')$ are subabundant, then 
\begin{align*}
    m(n_0+x+1)+t(n_0+1) & = m'(n'+x+1)+(t+m-m')(n'+1) \\ & \qquad + (m-m')(n_0-n'+x)+(t+m')(n_0-n') \\
    & \leq (n'+1)(\alpha+1) + (n_0-n')(\alpha+1) = (n_0+1)(\alpha+1),
\end{align*}
thus, $(n_0,m,t)$ is subabundant. An analogous proof holds for the superabundant case.

(2) By semicontinuity, in order to prove $T(n_0,s,t)$ it is enough to prove that the property holds for a \textit{special} choice of the points. Let $V_1$ be of dimension $(n'+1)$ and let $V_2$ be such that $V = V_1 \oplus V_2$. Let $Y_i = \bbP V_i \times X$ be the Segre product in $\bbP(V_i\otimes W)$, for $i = 1,2$. We specialize the $p_i$'s such that  $v_1,\ldots,v_{m'}$ are generic in $V_1$ and $v_{m'+1},\ldots,v_m$ are generic in $V_2$, then 
\[
    \widehat{T}_{p_i} Y  = (V_1 \oplus V_2) \otimes w_i + v_i \otimes \widehat{T}_{[w_i]} X = \left\{
\begin{array}{lll} 
  \widehat{T}_{p_i} Y_1  + V_2 \otimes w_i &\mbox{for each $i \in \{1, 2, \dots, m'\}$}, \\
  \widehat{T}_{p_i} Y_2  + V_1 \otimes w_i   & \mbox{for each $i \in \{s'+1, \dots, m\}$}. 
 \end{array}
 \right. 
\]
Thus, $\sum_{i=1}^m \widehat{T}_{p_i} Y  + \sum_{i=1}^t V\otimes w_i'$ is the direct sum 
\[
    \left(\sum_{i=1}^{m'} \widehat{T}_{p_i} Y_1 + \sum_{i=1}^t V_1 \otimes w_i' + \sum_{i=m'+1}^m V_1 \otimes w_i\right) \oplus \left(\sum_{i=m'+1}^{m} \widehat{T}_{p_i} Y_2+\sum_{i=1}^t V_2 \otimes w_i' + \sum_{i=1}^{m'} V_2 \otimes w_i\right). 
\] 
By the assumptions $T(n',m',t+m-m')$ and $T(n_0-n'-1,m-m',t+m')$, we have that both summands have the expected dimension and then also $T(n_0,m,t)$ holds. 
\end{proof}

%\begin{proposition}[Segre induction]\label{pro: segre induction} 
%    Let $X\subset \p^\alpha$ be a non-degenerate variety and let $Y$ be the image of $\p^n\times X$ under the embedding in $\p^{(\alpha+1)(n+1)-1}$ by the line bundle $\mathscr{O}_{\mathbb{P}V}(1) \boxtimes \mathscr{O}_X(1)$. Let
%    \[a_*=\rd{\frac{\alpha+1}{n+\dim(X)+1}}\mbox{ and } a^*=\ru{\frac{\alpha+1}{n+\dim(X)+1}}.\]
%    Suppose that $X$ is not $a_*$-defective nor $a^*$-defective. If $r\le (n+1)a_*$ or $r\ge (n+1)a^*$, then $Y$ is not $r$-defective.
%\end{proposition}

\begin{proof}[Proof of \Cref{pro:segre induction}]
    By \Cref{rmk:subabundant_superabundant}, it is enough to show that
    \begin{enumerate}
        \item $(n_0,(n_0+1)\, a_* ,0)$ is subabundant and $T(n_0,(n_0+1)\, a_* ,0)$ is true;
        \item $(n_0,(n_0+1)\, a^* ,0)$ is superabundant and $T(n_0,(n_0+1)\, a^* ,0)$ is true.
    \end{enumerate}
    % By \Cref{rmk:subabundant_superabundant}, this implies $T(n_0,m,0)$ is true for every $m \leq (n_0+1) \, a_*$. 
{We only prove the first statement because the proof of (2) is similar.}

    %(1) 
    Note that $(0,a_*,n_0 a_*)$ is subabundant by the definition of $a_*$. Moreover, since $T(0,a_*,0)$ is true by the assumption of non-defectivity of $X$, so is $T(0,a_*,n_0 a_*)$ {because adding generic points always impose the expected number of conditions}. Thus, by \Cref{lem:Segre_induction}, it is enough to prove that $(n_0-1,n_0 a_*, a_*)$ is subabundant and $T(n_0-1, n_0a_*, a_*)$ is true. 

    In order to prove this, we show that  $(n_0-i,(n_0-i+1)a_*, ia_*)$ is subabandant and $T(n_0-i,(n_0-i+1)a_*, ia_*)$ is true for all $i \in \{1,2,\ldots,n_0\}$. We proceed by backward induction on $i$. The case $i = n_0$ is true as commented above. If we assume that $(n_0-i,(n_0-i+1)a_*, ia_*)$ for any $i \in \{1,2,\ldots,n_0\}$ is subabandant and $T(n_0-i,(n_0-i+1)a_*, ia_*)$ then, by \Cref{lem:Segre_induction}, it follows that $(n_0-(i-1), (n_0-(i-1)+1)a_*, (i-1)a_*)$ is subabundant and $T(n_0-(i-1), (n_0-(i-1)+1)a_*, (i-1)a_*)$ holds. In particular, it holds for $i = 1$. 
%
 %   (2) A similar proof can be performed to conclude \ale{-- Should we add it?}. 
 \qedhere
    
% ---------

%     Note that $(0,a_*,n_0 a_*)$ is subabandant by the definition of $a_*$. Furthermore, since $T(0,a_*,0)$ is true by the assumption of non-defectivity of $X$, so is $T(0,a_*,n_0 a_*)$. Thus, by \Cref{lem:Segre_induction}, if $(n_0-1,n \, a_*, a_*)$ is subabandant and if $T(\ell-1, \ell\, a_*, a_*)$ is true, then $(\ell,s,t)$ is subabandant, and $T(\ell,s,t)$ is true. To prove this, we show the following two statements for each $i \in \{1, 2, \dots, \ell\}$; namely, (1) $(\ell-i,(\ell-i+1) \, a_*, i\, a_*)$ is subabandant, and (2) $T(\ell-i,(\ell-i+1) \, a_*, i\, a_*)$ is true.   

% The proof is by backward induction on $i$. We have already shown that (1) and (2) are true for $i = \ell$, so assume that $(\ell-i,(\ell-i+1) \, a_*, i\, a_*)$ is subabandant and that $T(\ell-i,(\ell-i+1) \, a_*, i\, a_*)$ is true for some $i \in \{1,2,\dots, \ell\}$. Since $(0,a_*,\ell a_*)$ is subabandant and since $T(0,a_*,\ell a_*)$ is true, it follows from Lemma~\ref{lem:Segre_induction} that $(\ell-(i-1),(\ell-i+2) \, a_*, (i-1) \, a_*)$ is subabandant and that $T(\ell-(i-1),(\ell-i+2) \, a_*, (i-1) \, a_*)$ is true. Therefore, $(\ell-1,\ell\, a_*, a_*)$ is subabandant, and $T(\ell-1,\ell\, a_*, a_*)$ is true. 

% Similarly, one can show that $(\ell,(\ell+1)\, a^* ,0)$ is superabandant and $T(\ell,(\ell+1)\, a^* ,0)$ is true. Thus, it follows from Remark~\ref{rem:abandance} that $T(\ell,s',0)$ for every $s \geq (\ell+1)\, a^*$.  
\end{proof}

% \begin{remark}
% Suppose that $a_* = a^*$. It follows from \Cref{pro:segre induction} that if $\sigma_{a_*}(X)$ has the expected dimension, then $\sigma_s(Y)$ has the expected dimension for any positive integer $s$. 
% \end{remark}

\begin{remark}
\label{rem:multiple}
Recall that $(n_0,(n_0+1)\, a_*,0)$ is subabundant and that $(n_0,(n_0+1)\, a^*, 0)$ is superabundant. Furthermore, 
\[
(n_0+1)\, a^* - (n_0+1)\, a_* = (n_0+1) (a^*-a_*) \leq n_0+1. 
\]
Thus, $(n_0+1) \, a_*$ is the greatest multiple of $n_0+1$ which is smaller than or equal to $\lfloor \frac{(n_0+1)(\alpha+1)}{n_0+\dim X+1} \rfloor$, while $(n_0+1) \, a^*$ is the least multiple of $n_0+1$ which is greater than or equal to $\lceil \frac{(n_0+1)(\alpha+1)}{n_0+\dim X+1}  \rceil$. 
Observe that the gap between the thresholds is $2|\bfn|+2$ in \Cref{BC}, while it is $n_0+1$ in our \Cref{thm:simultaneous_partialresult}.
\end{remark}

%\chiara{\begin{corollary}
%\label{cor:Segre_Veronese new}
%Let $k\ge 2$. Let $\bfn,\bfd \in \bbN^k$ be tuples of positive integers. If $d_i\ge 3$ for every %$i\in\{1,\dots,k\}$, then 
%$\SV_{(n_0,\bfn)}^{(1,\bfd)}$ is not $m$-defective for any
%\[  
%    m \le (n_0+1)\left\lfloor \frac{N_{\bfn,\bfd}}{n_0+|\bfn|+1}\right\rfloor  \quad \text{ or } \quad m \ge (n_0+1)\left\lceil \frac{N_{\bfn,\bfd}}{n_0+|\bfn|+1}\right\rceil.
%\]
%\end{corollary}

%\ale{-- add rmk to compare with TB-C -- }

%\chiara{\begin{remark}
%In particular, from the previous result it follows that $\SV_{(n_0,\bfn)}^{(1,\bfd)}$ is not $s$-defective for any
%\[  
%    s \le r_*((n_0,\bfn),(1,\bfd)) - n_0 \quad \text{ and } \quad s \ge r^*((n_0,\bfn),(1,\bfd)) +n_0.\footnote{or $+1$? to check}
%\]
%\end{remark}}

\begin{remark}If $X$ is a $d$-th Veronese embedding of $\mathbb{P}^{n_1}$ such that the Alexander-Hirschowitz Theorem implies that $\sigma_{a_*}(X)$ and $\sigma_{a^*}(X)$ have the expected dimensions, then by \Cref{pro:segre induction}, if $m\leq (n_0+1) \, a_*$ or $m \geq (n_0+1) \, a^*$, then $\sigma_m(Y)$ has the expected dimension. This gives an alternative proof to almost all cases of \cite[Corollary 2.2]{BCC}, and it extends it to any number of factors. %Thus, Remark~\ref{rem:multiple} implies that Proposition~\ref{pro:segre induction} gives an alternative proof of \cite[Corollary 2.2]{BCC}. 
\end{remark}
\appendix
\section{Numerical Computations}\label{appendix: contacci}

In this section we prove the numerical conditions that are needed in the main proofs. Let $k \geq 3$. Let $\bfn = (n_1 \leq n_2 \le\cdots \leq n_k)$ and $\bfd \succeq 3^k$ be $k$-tuples positive integers such that $n_1 \geq 2$. Let $s_r = s_r(\bfn,\bfd)$ and $\epsilon_r = \epsilon_r(\bfn,\bfd)$ be defined as in \Cref{thm: differential horace}.

% \begin{lemma}
%     \label{computation: N greater than dimX^2} 
%     Let $k\ge 2$. Let $n_2,\dots,n_k$ and $d_2,\dots,d_k$ be positive integers such that $d_i\ge 3$ for every $2\le i\le k$. Then 
%     \[
%         \prod_{i=2}^k\binom{n_i+d_i}{n_i}>(n_2+\dots+n_k)^2.
%     \]
%     \begin{proof}
%         Since the left-hand-side is monotonically increasing with respect to positive $d_i$'s, then it is enough to prove the statement for $d_2=\dots=d_k=3$. 
        
%         We want to reduce to the case $k=2$. We call
%         \[
%             f(k)=\prod_{i=2}^k\binom{n_i+3}{3}-(n_2+\dots+n_k)^2
%         \]
%         and we show that $f(k)\ge f(k-1)$ for every $k\ge 3$. Without loss of generality, we assume that $1 \le n_2\le\dots\le n_k$. Then
%         \begin{align*}
%         f(k)-f(k-1) &=\left( \binom{n_k+3}{3}-1\right) \prod_{i=2}^{k-1}\binom{n_i+3}{3}-n_k(n_k+2(n_2+\dots+n_{k-1}))\\
%         &\ge 4^{k-2}\left( \binom{n_k+3}{3}-1\right) -n_k(n_k+2(k-2)n_k)\\
%         &= \frac{4^{k-2}(n_k^3+6n_k^2+11n_k)}{6} - (2k-3)n_k^2\\
%         &> \frac{(2k-3)(n_k^3+6n_k^2)}{6} - (2k-3)n_k^2\ge (2k-3)n_k^2\left(\frac{n_k}{3}+1-1\right)\ge 0.
%         \end{align*}
%         Now, we only have to check that the statement holds for $k=2$, and this is immediate.
%     \end{proof}
% \end{lemma}

\begin{lemma}\label{lemma:d1=3_subabundant_BC}
If $r \in \{r^*(\bfn,\bfd),r_*(\bfn,\bfd)\}$, then $
        r-s_r \leq r_*(\bfn,\bfd(1)) - |\bfn| - 1$.
\end{lemma}
\begin{proof}
We prove that $r-s_r-r_*(\bfn,\bfd(1))+|\bfn|+1 \leq 0$. 
By definition of $s_r$ and $r_*(\bfn,\bfd(1))$, and by the fact that $-\lfloor \frac{a}{b} \rfloor \leq -\frac{a-b+1}{b}$, it suffices to show that
% Since 
% \[
% s_r  \geq \frac{(|\bfn|+1)r - N_{\bfn,\bfd(1)}-(|\bfn|-1)}{|\bfn|}
% \]
% and 
% \[
% r_*(\bfn,\bfd(1)) \geq \frac{N_{\bfn,\bfd(1)}-|\bfn|}{|\bfn|+1}, 
% \]
% it suffices to show that 
\[
    r - \frac{(|\bfn|+1)r - N_{\bfn,\bfd(1)}-(|\bfn|-1)}{|\bfn|}-\frac{N_{\bfn,\bfd(1)}-|\bfn|}{|\bfn|+1}+|\bfn|+1 \leq 0. 
\]
By clearing the denominators, one gets
\[
    -(|\bfn|+1)r + N_{\bfn,\bfd(1)}+|\bfn|^3+4 |\bfn|^2+|\bfn|-1 \leq 0. 
\]
Since $r \geq r_*(\bfn,\bfd)$ and by using again that $-\lfloor \frac{a}{b} \rfloor \leq -\frac{a-b+1}{b}$, it  is  enough to show that
\begin{equation}\label{eq:A2_1}
-N_{\bfn,\bfd}+N_{\bfn,\bfd(1)}+|\bfn|^3+4|\bfn|^2+2|\bfn|-1  \leq 0. 
\end{equation}
{Since
\[
    -N_{\bfn,\bfd}+N_{\bfn,\bfd(1)} =  -{n_1+d_1-1 \choose d_1}\prod_{i=2}^k {n_i+d_i \choose d_i} 
    % \leq -{n_1+2 \choose 3}\prod_{i=2}^k {n_i+3\choose 3}, 
\]
is decreasing with respect to $d_1,\ldots,d_k$, then it is enough to prove \eqref{eq:A2_1} for $\bfd = 3^k$. We do it by induction on $k$.

{\it Base case: we prove \eqref{eq:A2_1} for $\bfd = (3,3,3).$} Since $n_1 \leq n_2 \leq n_3$, it is enough to prove that 
\begin{equation}\label{eq:A2_1_k3}
    -{n_1 + 2 \choose 3}{n_1 + 3 \choose 3}{n_3 + 3 \choose 3} + (n_1+2n_3)^3 + 4(n_1+2n_3)^2 + 2(n_1+2n_3) - 1 \leq 0.
\end{equation}
As univariate polynomial in $\bbQ[n_1][n_3]$, the left-hand-side is equal to 
\begin{align*}
    & \left(-\frac{1}{216}n_1^{6}-\frac{1}{24}n_1^{5}-\frac{31}{216}n_1^{4}-\frac{17}{72}n_1^{3}-\frac{5}{27}n_1^{2}-\frac{1}{18}n_1+8\right)n_3^{3}\\
    & +\left(-\frac{1}{36}n_1^{6}-\frac{1}{4}n_1^{5}-\frac{31}{36}n_1^{4}-\frac{17}{12}n_1^{3}-\frac{10}{9}n_1^{2}+\frac{35}{3}n_1+16\right)n_3^{2} \\
    & +\left(-\frac{11}{216}n_1^{6}-\frac{11}{24}n_1^{5}-\frac{341}{216}n_1^{4}-\frac{187}{72}n_1^{3}+\frac{107}{27}n_1^{2}+\frac{277}{18}n_1+4\right)n_3 \\
    & -\frac{1}{36}n_1^{6}-\frac{1}{4}n_1^{5}-\frac{31}{36}n_1^{4}-\frac{5}{12}n_1^{3}+\frac{26}{9}n_1^{2}+\frac{5}{3}n_1-1.
\end{align*}
It is immediate to note that all coefficients are negative for $n_1 \geq 3$, allowing us to conclude that \eqref{eq:A2_1_k3} holds, and consequently \eqref{eq:A2_1} for $\bfd = (3,3,3)$ and $n_1 \geq 3$.

We are left with the case $n_1 = 2$ for which \eqref{eq:A2_1_k3} doesn't hold for $n_3 \gg 0$. Therefore, we prove directly \eqref{eq:A2_1} by substituting $n_1 = 2$, i.e., we consider 
\begin{equation}\label{eq:A2_1_k3'}
    -4{n_2 + 3 \choose 3}{n_3 + 3 \choose 3} + (2+n_2+n_3)^3 + 4(2+n_2+n_3)^2 + 2(2+n_2+n_3) - 1 \leq 0.
\end{equation}
As a univariate polynomial in $\bbQ[n_2][n_3]$, the left-hand-side is equal to 
\begin{align*}
    &\left(-\frac{1}{9}n_2^{3}-\frac{2}{3}n_2^{2}-\frac{11}{9}n_2+\frac{1}{3}\right)n_3^{3}+\left(-\frac{2}{3}n_2^{3}-4\,n_2^{2}-\frac{13}{3}n_2+6\right)n_3^{2}\\
    &+\left(-\frac{11}{9}n_2^{3}-\frac{13}{3}n_2^{2}+\frac{59}{9}n_2+\frac{68}{3}\right)n_3+\frac{1}{3}n_2^{3}+6\,n_2^{2}+\frac{68}{3}n_2+23
\end{align*}
For $n_2\ge 2$ the first and the second coefficients are positive and the fourth is negative, so there is only one change of sign in the coefficients. Hence, by Descartes' rule of signs it has only one positive real root. In order to show that \eqref{eq:A2_1_k3'} holds for every $n_3\ge 2$, it is enough to show that such polynomial is negative for $n_3=2$. For $n_3 = 2$ it becomes 
\[
   % -\frac{1}{9}n_2^{6}-\frac{4}{3}n_2^{5}-\frac{58}{9}n_2^{4}-8\,n_2^{3}+\frac{167}{9}n_2^{2}+\frac{136}{3}n_2+23.
-\frac{17}{3}{n_2}^{3}-24{n_2}^{2}+\frac{26}{3}{n_2}+95\le0.
\]
Hence \eqref{eq:A2_1} holds for $\bfd = (3,3,3)$ and $n_1 = 2$. 

{\it Inductive step: we prove \eqref{eq:A2_1} for $\bfd = 3^k$ and $k\ge4$.} %$\bfd = 3^k$ and $k\ge4$, assuming it for $\bfd = 3^{k-1}$.} 
%Assume $k \geq 4$. 
Let $\bfn' = (n_1,\ldots,n_{k-1})$.
By inductive assumption%, the left-hand-side of \eqref{eq:A2_1} is at most 
\begin{align*}-\binom{n_1+2}{3}&\prod_{i=2}^{k-1}\binom{n_i+3}{3}\binom{n_k+3}{3}+ (|\bfn|^3+4|\bfn|^2+2|\bfn|-1)\\
&\le
    -{n_k + 3 \choose 3}(|\bfn'|^3+4|\bfn'|^2+2|\bfn'|-1) + (|\bfn|^3+4|\bfn|^2+2|\bfn|-1). 
\end{align*}
We express this as univariate polynomial in $\bbQ[|\bfn'|][n_k]$:
\begin{align*}
    &\left(-\frac{1}{6}|\bfn'|^{3}-\frac{2}{3}|\bfn'|^{2}-\frac{1}{3}|\bfn'|+\frac{7}{6}\right)n_k^{3}\\
    &+\left(-|\bfn'|^{3}-4\,|\bfn'|^{2}+|\bfn'|+5\right)n_k^{2}\\
    &+\left(-\frac{11}{6}|\bfn'|^{3}-\frac{13}{3}|\bfn'|^{2}+\frac{13}{3}|\bfn'|+\frac{23}{6}\right)n_k.
\end{align*}
Since $2 \leq n_1 \leq n_2 \leq \cdots \leq n_{k-1}$ and $k \geq 4$, then $|\bfn'| \geq 6$: under this condition all coefficients of the latter polynomial are negative and then \eqref{eq:A2_1} holds also for $\bfd = 3^k$ for any $k \geq 4$.}
\end{proof}

\begin{lemma}\label{lemma:s_vs_epsilon}
If $r \in \{r^*(\bfn,\bfd),r_*(\bfn,\bfd)\}$, then $s_r \geq \epsilon_r$.
\end{lemma}
\begin{proof}
{Note that 
    \begin{align*}
        s_r - \epsilon_r & = s_r - (|\bfn|+1)r + N_{\bfn,\bfd(1)} + |\bfn|s_r = (|\bfn|+1)(s_r-r) + N_{\bfn,\bfd(1)} \\ 
        & \geq (|\bfn|+1)(-r_*(\bfn,\bfd(1))+|\bfn|+1) + N_{\bfn,\bfd(1)} \\
        & \geq (|\bfn|+1)\left(-\frac{N_{\bfn,\bfd(1)}}{|\bfn|+1}+|\bfn|+1\right) + N_{\bfn,\bfd(1)} = (|\bfn|+1)^2 
    \end{align*}
    where the first inequality is \Cref{lemma:d1=3_subabundant_BC} and the second one follows by definition of $r_*$.}
\end{proof}

\begin{lemma}\label{lemmaBC} 
If $r \in \{r^*(\bfn,\bfd),r_*(\bfn,\bfd)\}$, then $r^*(\bfn,\bfd(2))+|\bfn|+1 \le r-s_r-\epsilon_r$.
\end{lemma}
\begin{proof}
By definition $\epsilon_r\le |\bfn|-1$, so $r-s_r-\epsilon_r-r^*(\bfn,\bfd(2))-|\bfn|-1\ge r-s_r-r^*(\bfn,\bfd(2))-2|\bfn|$. We prove that the latter is greater or equal than zero. By definition of $s_r$
\begin{align*}
    r-s_r-r^*(\bfn,\bfd(2))-2|\bfn| \ge r-\frac{(|\bfn|+1)r-N_{\bfn,\bfd(1)}}{|\bfn|}-r^*(\bfn,\bfd(2))-2|\bfn|.
\end{align*}
% \begin{align*}
%     r-s_r-\epsilon_r-&r^*(\bfn,\bfd(2))-|\bfn|-1\ge r-s_r-r^*(\bfn,\bfd(2))-2|\bfn|\\
%     &=r-\rd{\frac{(|\bfn|+1)r-N_{\bfn,\bfd(1)}}{|\bfn|}}-r^*(\bfn,\bfd(2))-2|\bfn|\\
%     &\ge r-\frac{(|\bfn|+1)r-N_{\bfn,\bfd(1)}}{|\bfn|}-r^*(\bfn,\bfd(2))-2|\bfn|.
% \end{align*}
Clearing the denominator, it suffices to show that
\[
    -r+N_{\bfn,\bfd(1)}-|\bfn|r^*(\bfn,\bfd(2))-2|\bfn|^2 \geq 0.
\]
Since $r\le r^*(\bfn,\bfd)$, then
\begin{align*}
    -r+N_{\bfn,\bfd(1)}-|\bfn|r^*(\bfn,\bfd(2))&-2|\bfn|^2 
     \geq -\ru{\frac{N_{\bfn,\bfd}}{|\bfn|+1}}+N_{\bfn,\bfd(1)}-|\bfn|\ru{\frac{N_{\bfn,\bfd(2)}}{|\bfn|+1}}-2|\bfn|^2\\
    & \ge -\frac{N_{\bfn,\bfd}+|\bfn|}{|\bfn|+1}+N_{\bfn,\bfd(1)}-|\bfn|\frac{N_{\bfn,\bfd(2)}+|\bfn|}{|\bfn|+1}-2|\bfn|^2.
\end{align*}
Clearing the denominator, we are left to prove that 
\begin{equation}\label{eq:A3_1}
    -N_{\bfn,\bfd}+(|\bfn|+1)N_{\bfn,\bfd(1)}-|\bfn|N_{\bfn,\bfd(2)}-|\bfn|(|\bfn|+1)(2|\bfn|+1) \geq 0
\end{equation}
Observe that
\[
    -N_{\bfn,\bfd}+(|\bfn|+1)N_{\bfn,\bfd(1)}-|\bfn|N_{\bfn,\bfd(2)} =\frac{(n_1+d_1-2)!}{(n_1-1)!d_1!}\prod_{i=2}^k\binom{n_i+d_i}{n_i}\left(|\bfn|d_1-n_1-d_1+1
    \right). 
\]
The left-hand-side of \eqref{eq:A3_1} is increasing when the $d_i$'s are positive and increasing. Therefore, it is enough to prove it for $\bfd = 3^k$. We do it by induction on $k$.

{\it Base case: we prove \eqref{eq:A3_1} for $\bfd = (3,3,3)$.} We employ the fact that $n_1\le n_2\le n_3$ to deduce that the left-hand-side of \eqref{eq:A3_1} for $\bfd = (3,3,3)$ is greater or equal to 
\begin{align*}
\frac{1}{3}\binom{n_1+1}{2}&\binom{n_1+3}{3}\binom{n_3+3}{3}\left(5n_1+3n_3-2
\right)-(n_1+2n_3)(n_1+2n_3+1)(2n_1+4n_3+1)\\
=&\left(
\frac{1}{72}n_{1}^{5}+\frac{7}{72}n_{1}^{4}+\frac{17}{72}n_{1}^{3}+\frac{17}{72}n_{1}^{2}+\frac{1}{12}n_{1}\right
)n_{3}^{4}\\
&+\left(
%\frac{5}{216}n_{1}^{6}+\frac{17}{72}n_{1}^{5}+\frac{197}{216}n_{1}^{4}+\frac{119}{72}n_{1}^{3}+\frac{151}{108}n_{1}^{2}+\frac{4}{9}n_{1}
\frac{5}{216}n_{1}^{6}+\frac{17}{72}n_{1}^{5}+\frac{197}{216}n_{1}^{4}+\frac{119}{72}n_{1
      }^{3}+\frac{151}{108}n_{1}^{2}+\frac{4}{9}n_{1}-16
\right)n_{3}^{3}\\
&+\left(
%\frac{5}{36}n_{1}^{6}+\frac{77}{72}n_{1}^{5}+\frac{73}{24}n_{1}^{4}+\frac{289}{72}n_{1}^{3}+\frac{179}{72}n_{1}^{2}+\frac{7}{12}n_{1}-16
\frac{5}{36}n_{1}^{6}+\frac{77}{72}n_{1}^{5}+\frac{73}{24}n_{1}^{4}+\frac{289}{72}n_{1}^{3}+\frac{179}{72}n_{1}^{2}-\frac{281}{12}n_{1}-12
\right)n_{3}^{2}\\
&+\left(
%\frac{55}{216}n_{1}^{6}+\frac{127}{72}n_{1}^{5}+\frac{907}{216}n_{1}^{4}+\frac{289}{72}n_{1}^{3}+\frac{131}{108}n_{1}^{2}-\frac{163}{9}n_{1}-6
\frac{55}{216}n_{1}^{6}+\frac{127}{72}n_{1}^{5}+\frac{907}{216}n_{1}^{4}+\frac{289}{72}n_{1}^{3}-\frac{1\,165}{108}n_{1}^{2}-\frac{109}{9}n_{1}-2
\right)n_{3}\\
&+
%\frac{5}{36}n_{1}^{6}+\frac{11}{12}n_{1}^{5}+\frac{71}{36}n_{1}^{4}+\frac{17}{12}n_{1}^{3}-\frac{46}{9}n_{1}^{2}-\frac{10}{3}n_{1}.
\frac{5}{36}n_{1}^{6}+\frac{11}{12}n_{1}^{5}+\frac{71}{36}n_{1}^{4}-\frac{7}{12}n_{1}^{3
      }-\frac{28}{9}n_{1}^{2}-\frac{4}{3}n_{1}. 
\end{align*}
Regarding it as a univariate polynomial in $\Q[n_1][n_3]$, we observe that each coefficient is positive under our assumption that $n_1\ge 2$. Hence, \eqref{eq:A3_1} holds for $\bfd = (3,3,3)$. 

{\it Inductive step: we prove \eqref{eq:A3_1} for $\bfd = 3^k$ with $k \geq 4$.} Let $\bfn' = (n_1,\ldots,n_{k-1})$. By inductive assumption, 
\begin{align*}
    \frac{1}{3}{n_1 + 1 \choose 2}\prod_{i=2}^k\binom{n_i+3}{n_i}&\left(3|\bfn|-n_1-2
    \right) - |\bfn|(|\bfn|+1)(2|\bfn|+1) \\ & \geq {n_k + 3 \choose n_k}|\bfn'|(|\bfn'|+1)(2|\bfn'|+1) - |\bfn|(|\bfn|+1)(2|\bfn|+1)
\end{align*}
We express this as univariate polynomial in $\bbQ[|\bfn'|][n_k]$: 
\begin{align*}
    &\left(\frac{1}{3}|\bfn'|^{3}+\frac{1}{2}|\bfn'|^{2}+\frac{1}{6}|\bfn'|-2\right)n_k^3 \\
    & \qquad + \left(2\,|\bfn'|^{3}+3\,|\bfn'|^{2}-5\,|\bfn'|-3\right)n_k^2 \\
     & \qquad + \left(\frac{11}{3}|\bfn'|^{3}-\frac{1}{2}|\bfn'|^{2}-\frac{25}{6}|\bfn'|-1\right)n_k.
\end{align*}
Since $2 \leq n_1 \leq n_2 \leq \cdots \leq n_{k-1}$ and $k \geq 4$, we have $|\bfn'| \geq 6$: under this condition all coefficients of the latter polynomial are negative, and hence \eqref{eq:A3_1} also holds for $\bfd = 3^k$ for any $k \geq 4$. \qedhere
\end{proof}

\bibliographystyle{alpha}
\bibliography{SVreferences.bib}

\end{document}